\documentclass[12pt]{article}
\usepackage{geometry}                
\usepackage[longnamesfirst]{natbib}
\bibpunct[ ]{(}{)}{,}{a}{}{,}

\usepackage{graphicx}
\usepackage{amssymb}
\usepackage{amsmath}
\usepackage{amsthm}
\usepackage{amsfonts}
\usepackage{amsmath,amsfonts, amssymb, amsthm, amscd,graphicx}
\usepackage{latexsym}
\newcommand{\R}{\ensuremath{\mathbb{R}}}

 \newcommand{\N}{\ensuremath{\mathbb{N}}}

\newcommand\sa{\sqrt{a}}

\newcommand\EE{{\mathbb E}}
\newcommand\PP{{\mathbb P}}
 \newcommand{\eps}{\varepsilon}
\theoremstyle{plain}
 \newtheorem{theorem}{Theorem}
 \newtheorem{lemma}[theorem]{Lemma}
 \newtheorem{propo}[theorem]{Proposition}
 \newtheorem{corollary}[theorem]{Corollary}

\theoremstyle{definition} 
 \newtheorem{rem}[theorem]{Remark}

\title{The right tail exponent of the Tracy-Widom-$\beta$ distribution}
\author{Laure Dumaz \and B\'alint Vir\'ag}

\begin{document}

\maketitle

\begin{abstract}
The Tracy-Widom $\beta$ distribution is the large
dimensional limit of the top eigenvalue of $\beta$ random
matrix ensembles.  We use the stochastic Airy operator
representation to show that as $a\to\infty$ the tail of the
Tracy Widom distribution satisfies
$$P \left(TW_{\beta} > a \right) = a^{-\frac34
\beta+o(1)}\exp\left(-\frac {2} {3} \beta a^{3/2} \right).
$$
\end{abstract}

\begin{center}
\includegraphics[width = 15cm]{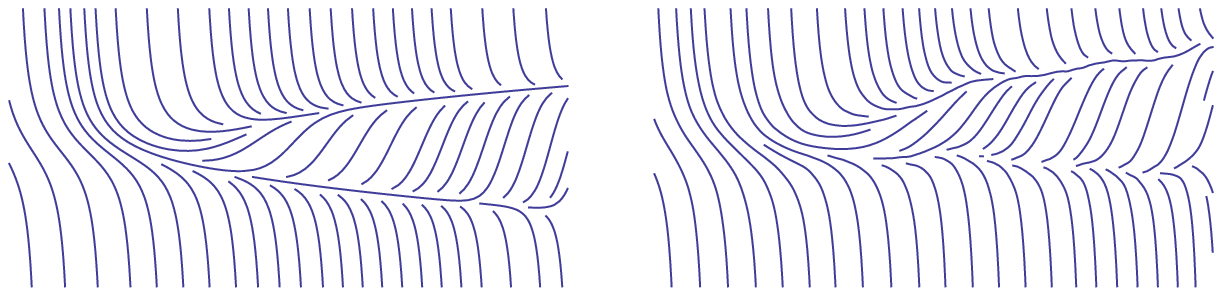}
\\ \it Flowlines for the ODE of
$X-\frac2{\sqrt{\beta}}B$
\end{center}
\addtocounter{figure}{1}

\section{Introduction}
For $\beta > 0$ fixed, we examine the probability density of
$\lambda_1 \geq \lambda_2 \geq \ldots \geq \lambda_n \in
\R$ given by :
\begin{align} \label{loivp}
\mathbb{P}_{\beta}(\lambda_1,\lambda_2,\ldots,\lambda_n) =
\frac{1}{Z_{n,\beta}} e^{- \beta \sum_{k=1}^n
\lambda_k^2/4} \prod_{j<k} |\lambda_j-\lambda_k|^{\beta},
\end{align}
in which $Z_{n,\beta}$ is a normalizing constant. This
family of distribution is called the $\beta$-ensemble. When
$\beta = 1,2 \mbox{ or } 4$, this is the joint density of
eigenvalues for respectively the Gaussian orthogonal,
unitary, or symplectic ensembles of random matrix theory.
But the law $(\ref{loivp})$ has a physical sense for the
$\beta$ as it describes a one-dimensional Coulomb gas at
inverse temperature $\beta$. \cite{DE} discovered a that
\eqref{loivp} is the eigenvalue distribution for the
tridiagonal matrix
\begin{align*}
H_n^{\beta} = \frac{1}{\sqrt{\beta}} \left [
\begin{array}{ccccc}
  g_1 & \chi_{(n-1)\beta} & \quad & \quad &  \\
  \chi_{(n-1)\beta} & g_2 & \chi_{(n-2)\beta} & \quad &  \\
 \quad  & \ddots & \ddots & \ddots &  \\
 \quad  & \quad & \chi_{2\beta} & g_{n-1} & \chi_{\beta} \\
 \quad  & \quad & \quad & \chi_{\beta} & g_n
\end{array}
\right ],
\end{align*}
where the random variables $g_1, g_2, \cdots, g_n$ are
independent Gaussians with mean $0$ and variance $2$ and
$\chi_{\beta},\chi_{2\beta},\cdots,\chi_{(n-1)\beta}$ are
independents $\chi$ random variables indexed by the shape
parameter.

When $n \uparrow \infty$, the largest eigenvalue centered
by $2\sqrt{n}$ and scaled by $n^{1/6}$ converges in law to
the Tracy-Widom($\beta$) distribution. This was first shown
in  \cite{TW} and \cite{TW2} for the cases $\beta=1,2$ or
$4$, where exact formulae are available. \cite{balint}
extended this result for all the $\beta$. They show that
the rescaled rescaled operator:
\begin{align*}
\tilde{H}_n^{\beta} := n^{1/6}(2\sqrt{n} I - H_n^{\beta})
\end{align*}
converges to the stochastic Airy operator ($SAE_{\beta}$) :
\begin{align}\label{SAE}
\mathcal{H}_{\beta} = - \frac{d^2}{dx^2} + x +
\frac{2}{\sqrt{\beta}}\, B'_x
\end{align}
in the appropriate sense (here $B'$ is a white noise). In
particular, the low-lying eigenvalues of
$\tilde{H}_n^{\beta}$ converge in law to those of
$\mathcal{H}_{\beta}$. Thanks to the Ricatti transform, the
eigenvalues of $SAE_{\beta}$ can be reinterpreted in terms
of the explosion probabilities of a one-dimensional
diffusion. In particular, \cite{balint} show that
\begin{eqnarray}\label{th1}
\PP(TW_{\beta} > a) &=& \PP_{+\infty}(X \mbox{ blows up  in
a finite time}),
\end{eqnarray}
where $X$ is the diffusion
\begin{equation}
\label{SDE1} \left\{
\begin{array}{l}
dX(t) = (t + a - X^2(t)) dt + \frac{2} {\sqrt{\beta}}
dB(t),
\\ X(0)=\infty.
\end{array} \right.
\end{equation}

Note also that  $X-\frac{2} {\sqrt{\beta}} B$ satisfies an
ODE, simulated on the front page with $\beta=\infty$ and
$2$. The starting time of the separatrix is distributed as
$-$TW$_\beta$.

Asymptotic expansions of beta-ensembles are of active interest in the literature, see for example \cite{CEM}, \cite{Balint2}, 
\cite{Dyson} and \cite{ChenManning}.

In this article, we study the diffusion $\eqref{SDE1}$ in
order to obtain the right tail of the Tracy-Widom law. Our
main tool will be the Cameron-Martin-Girsanov theorem: it
permits us to change the drift coefficient of the diffusion
and evaluate the probability of explosion using
the new process.

Using the variational characterization of the eigenvalues
of SAE$_{\beta}$ $(\ref{SAE})$ and an analysis of the SDE
\eqref{SDE1} \cite{balint} show that as  $a\to\infty$ we
have
\begin{align}
\begin{array}{lll}
P\left(TW_{\beta} < - a\right) &=& \exp\left(-\frac{1}{24} \beta a^3 (1+o(1))\right),\quad \mbox{and} \\
P\left(TW_{\beta} > a\right) &=& \exp\left(-\frac{2}{3}
\beta a^{3/2} (1+o(1))\right).
\end{array}
\end{align}
While we were finishing this article, \cite{Celine}, in a
physics paper, using completely different methods,
calculated more precise asymptotics for the {\it left tail}
of the Tracy-Widom distribution.

In this paper we evaluate the exponent of the polynomial
factor in the asymptotics of the right tail.

\begin{theorem}\label{main}
When $a \to + \infty$, we have
\begin{align}\label{Statement}
P \left(TW_{\beta} > a \right) = a^{-3\beta/4
}\exp\left(-\frac {2} {3} \beta a^{3/2}+ O\big(\sqrt{\ln
a}\big)\right).
\end{align}
\end{theorem}

This generalizes, in a less precise form, a result that
follows from Painlev\'e asymptotics for the case $\beta =
2$ (see the slide $3$ of the presentation of \cite{Baik}).
\begin{align*}
P(TW_{2} > a) = \frac{a^{-3/2}}{16 \pi}\exp\left(-\frac{4}{3} a^{3/2} +
O\left({a^{-3/2}}\right)\right).
\end{align*}

The structure of the proof of Theorem \ref{main} is
contained in Section \ref{s:outline}.

\paragraph{Preliminaries and notation.}
For every initial condition in $[-\infty, +
\infty]$, the SDE $(\ref{SDE1})$ admits a unique solution,
and this solution is increasing in $a$ for each time $t$
(see Fact 3.1 in \cite{balint}). From now on, we denote by
$(\Omega, \mathcal{F}, \PP_{(t,x)})$ the probability space
on which the solution of this SDE $X$ begins at time $t$
with the value $X_t = x$ almost surely, and $\EE_{(t,x)}$
its corresponding expectation ($x \in [-\infty,+\infty]$).
When the starting time is $t=0$, we simply write $\PP_{x}$
and $\EE_{x}$.

\medskip

\noindent The  first passage time to a
level $x \in [-\infty, \infty]$ for the diffusion $X$ will be denoted $T_x
:= \inf\{s \geq 0, \; X_s = x\}$.

\medskip

Throughout this paper, we study many solutions of
stochastic differential equations by comparing them to
expressions involving Brownian motion. The letter $B$
will denote a standard Brownian motion on the probability
space $(\Omega, \mathcal{F}, P)$. We will use the following
easy estimates. For the upper bounds, these inequalities
hold for every $x \geq 0$:
\begin{align}\label{TailBMup}
\left\{
\begin{array}{ll}
P(B_1 > x) &\leq e^{-\frac{1}{2} x^2} \\
P\left(\sup_{t \in [0,1]} |B_t| > x\right) &\leq 4
e^{-\frac{1}{2} x^2}.
\end{array}\right.
\end{align}
\noindent For the lower bounds, there exists $c_{bm} > 0$
such that for every  $\eps \in (0,1)$:
\begin{align}\label{TailBMlower}
P\left(\sup_{t \in [0,1]} |B_t| < \eps\right) \geq
\exp\left(-c_{bm}\: \frac{1}{\eps^2}\right).
\end{align}
In the sequel, asymptotic notation always refers to
$a\to\infty$ unless stated otherwise. Inequalities are
meant to hold for all large enough $a$.

\section{Proof of the Theorem \ref{main}}\label{s:outline}
\begin{figure}[!h]
\centering
\includegraphics[width = 10cm]{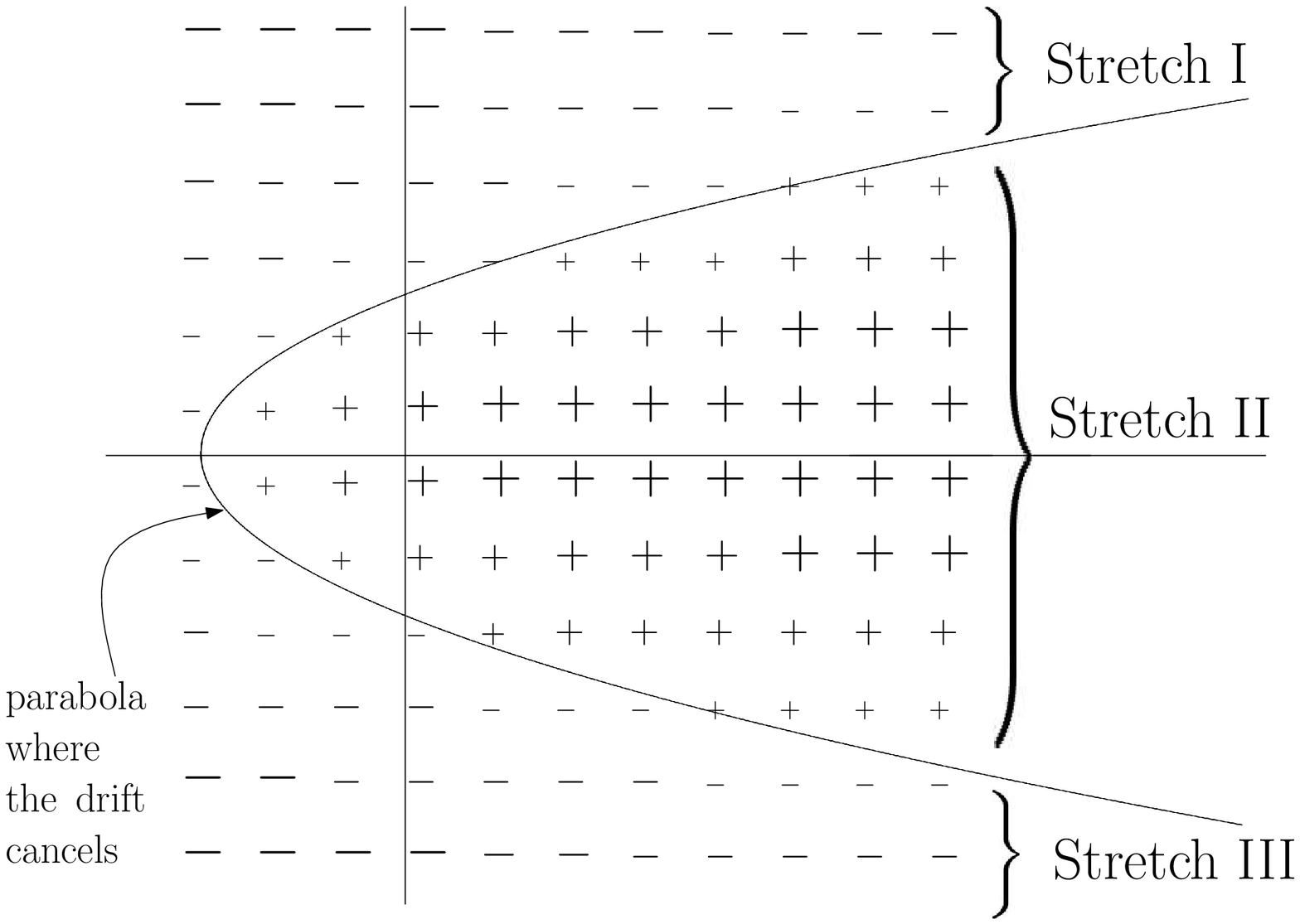}
\caption{\label{picparabola}The critical parabola $\{(t,x) : t+ a -x^2 = 0\}$}
\end{figure}

This section gives the structure of proof of the main
theorem. The proof of technical points will be treated in
the following sections in chronological order.

We rely on the characterization \eqref{th1}, and separate
our study of the diffusion \eqref{SDE1} into three
distinguished parts demarcated by the {\bf critical
parabola}
$$\{(t,x) : t+ a -x^2 = 0\}$$ where the drift vanishes (see
Figure \ref{picparabola}). The exponential leading term of
the asymptotic (\ref{Statement}) comes from the part inside
the parabola (Stretch II): the drift is positive and makes
it difficult for the particle to go down. One part of the
logarithmic term comes from the time it takes to reach the
upper part of the parabola (Stretch I): the $t$-term of the
drift adds this cost.

\subsection{Upper bound, above the parabola}

At first, let us approximate the critical parabola by the
two horizontal lines $\sqrt{a}$ and $-\sqrt{a}$ (as the
blow-down times will be typically very small). Moreover,
the part below the parabola gives no contribution for the
upper bound, and we use
\begin{figure}[!b]
\centering
\includegraphics[width = 10cm]{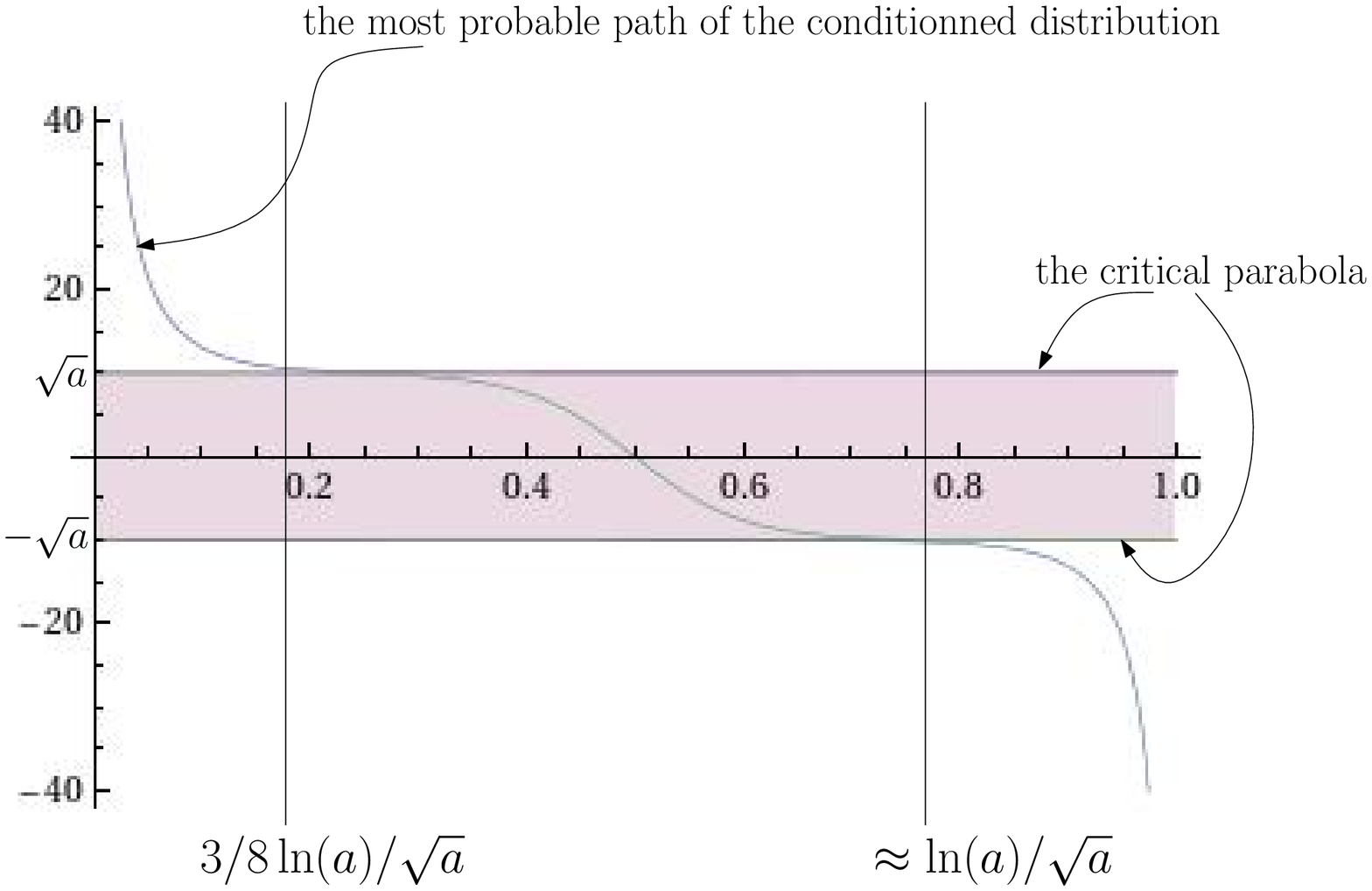}
\caption{\label{picmostprobablepath}The most probable path of the conditioned
diffusion for $a=100$}
\end{figure}
$$\PP_{\infty}(T_{-\infty} < +\infty) \leq \PP_{\infty}(T_{-\sqrt{a}} < +\infty).$$
The first step is to control the time it takes to reach
$\sqrt{a}$. Indeed, as the cost for crossing the interval
$[-\sqrt{a}, \sqrt{a}]$ increases with time, we need to
find a good lower bound for this time. A comparison with
the solution of an ODE linked to our SDE enables us
to have a quite precise information: its typical value is
$3/8 \ln a/\sqrt{a}$, which does not depend on the factor
$\beta$.  It is very unlikely to happen in time faster than
$$\tau_-=(3/8 - 1/\sqrt{\ln a}) \ln a/\sqrt{a}.$$ This is the
content of Proposition \ref{resuabove}. Therefore, using
this proposition, the decreasing property in $t$ and the
Markov property, we can write:
\begin{align}
\PP_{\infty}\left(T_{\sqrt{a}} <\tau_-,\;  T_{-\sqrt{a}} < \infty \right)
\leq \exp\left(-\frac{4}{3} \beta e^{2 \sqrt{\ln
a}}\right)\, \PP_{\sqrt{a}}(T_{-\sqrt{a}} < \infty).
\label{startimm}
\end{align}

The asymptotic formula \eqref{formulelessprecise} given by
Lemma \ref{lessprecise} will highlight the fact that even
if the process is considered to start immediately at
$\sqrt{a}$ in line \eqref{startimm}, the award is small (of
the order $\exp(O(\ln a)$) compared to the cost it takes to
go down quickly. Consequently, with a much more significant
probability, it will take a longer time than the one
considered in (\ref{startimm}) to reach $\sqrt{a}$. Let us
find an upper bound for this case:
\begin{align*}
\PP_{\infty}\left(T_{-\sqrt{a}} < \infty,\; T_{\sqrt{a}}
\geq \tau_-\right) \leq \PP_{\tau_-,
\sqrt{a}}(T_{-\sqrt{a}} < \infty)
\end{align*}
Thanks to the Markov property, the process $X$ under the
probability measure $\PP_ {\tau_-, \sqrt{a}}$ is
identically distributed with $\tilde{X}$ defined with the
same SDE (\ref{SDE1}) where the variable $a$ is replaced by
$\tilde{a}:= a + \tau_-$ with the initial condition
$\tilde{X}(0) = \sqrt{a}$. Observe now $\sqrt{a} <
\sqrt{\tilde{a}}$, but it does not matter as we will be
allowed to reduce the interval $[-\sqrt{a},\sqrt{a}]$ a bit
without affecting the relevant terms in our asymptotics.
More precisely, the interval we will study for the middle
is  $[-\sqrt{a} + \delta, \sqrt{a} - \delta]$, where
$\delta := \sqrt[4]{\ln a/a}$.

\noindent Let $\tilde{T}_x$ denote the first passage times
of $\tilde{X}$.  The inequality $\sqrt{\tilde{a}} -
\tilde{\delta} \leq \sqrt{a}$ gives:
\begin{align}
\PP_{\tau_-, \sqrt{a}} (T_{-\sqrt{a}} < \infty)
= \PP_{\sqrt{a}}(\tilde{T}_{-\sqrt{a}} < \infty) \leq \PP_{\sqrt{\tilde{a}} -
\tilde{\delta}}(\tilde{T}_{-\sqrt{\tilde{a}} +
\tilde{\delta}} < \infty). \label{majodebut}
\end{align}

\subsection{Preliminary upper bound inside the parabola}

Recall $\delta := \sqrt[4]{\ln a/a}$ we would
like to find the asymptotic of:
\begin{align}\label{studiedprobamiddle}
 \PP_{\sqrt{a}- \delta}\left(T_{-\sqrt{a} + \delta} < \infty\right) \geq \PP_{\sqrt{a}}\left(T_{-\sqrt{a}} <
\infty\right).
\end{align}

\noindent The key is Girsanov formula.

\paragraph{Girsanov formula.}
To evaluate the cost inside the parabola, we use the
Cameron-Martin-Girsanov formula which allows us to change the
drift coefficient of $X$ and evaluate the relevant
probability by analyzing the new process. The issue is to find a
suitable new drift. The best would be to have the one
corresponding to the conditional distribution of the
diffusion $X$ under the event it crosses the critical
parabola in a finite time: this would lead to an exact
formula. As we are not able to do that, we use an
approximation of the conditional diffusion. In this
direction, we introduce a new SDE in which the drift of $X$
is reversed with a correction term given by a function
$\varphi$:
\begin{align*}
dY_t = \left(-a + Y_t^2 - t + \varphi(Y_t)\right) dt +
\frac{2}{\sqrt{\beta}} dB_t.
\end{align*}
Let  $T'_x$ denote the first passage time of the process $Y$
to the level $x$.

With this diffusion and under some mild assumptions, the
Cameron-Martin Girsanov formula gives for every non
negative measurable function $f$ and every fixed time $t > 0$ and
level $l \in [0,1]$:
\begin{align}
\EE_{\sqrt{a} - l}\left(f\left(X_u, u \leq
T_{-\sqrt{a}+l}\wedge t\right)\right) = \EE_{\sqrt{a} -
l}\left(f\left(Y_u, u \leq T'_{-\sqrt{a}+l}\wedge
t\right)\;\exp\big(G_{T'_{-\sqrt{a}+l} \wedge
t}(Y)\big)\right).
\end{align}
More details about this and the application of the Girsanov
formula can be found in Section \ref{subsubsec:Girsanovformula}.
The exact expression of
$G_{T'_{-\sqrt{a}+l}}(Y)$ contains $\beta/4$ times
\begin{align}- \frac{8}{3} a^{3/2} - \frac{4}{3} l^3 + 4
\sqrt{a}\, l^2 + 2 l\, T'_{-\sqrt{a}+l} - 2 \sqrt{a}\,
T'_{-\sqrt{a}+l} + \left(\frac{8}{\beta} -2\right)
\int_0^{T'_{-\sqrt{a}+l}} Y_t dt. \label{eqstar}
\end{align}
plus terms involving the function $\varphi$.

Notice that we can already see the correct coefficient in
front of the main term. We are now confronted with an
expectation over the paths of the diffusion $Y$. To find a
good estimate of the exponential martingale, we need to
control the first passage time to the level $-\sqrt{a}+l$
and check that the diffusion do not go far above $\sqrt{a}$
when this time is finite (in order to control the last integral in \eqref{eqstar}).  We will at first focus on a
preliminary bound, for which we do not need any result
about the first passage time. The price of this approach is
that it uses a finer control of the paths which go down.

\paragraph{Control of the paths.}

\noindent To have a good control of the paths, we examine
at first a smaller interval than
$[\sqrt{a}-\delta,-\sqrt{a}+\delta]$. On this new interval,
the diffusion will go down without hitting $\sqrt{a}$ with
a sufficiently large probability. Indeed, we show that for
$\eps := \frac4{\sqrt{\beta}}\sqrt{\ln a}/\sqrt[4]{a}$ we
have
\begin{equation}
\label{pathcontrol}
     \PP_{\sqrt{a}-\eps} (T_{-\sqrt{a} + \eps} < \infty)
     =(1-o(1))
\PP_{\sqrt{a}-\eps} \Bigl( T_{-\sqrt{a} + \eps} < \infty,
T_{-\sqrt{a}+\eps} < T_{\sqrt{a}-\eps/2}\Bigr).
\end{equation}
This is accomplished by two applications of the strong
Markov property. From now on, we denote $T_+ =
T_{\sqrt{a}-\eps/2}$, $T_- = T_{-\sqrt{a} + \eps}$ and
$\mathcal{A} = \{  T_{+} >  T_{-} \}$, and have:
\begin{eqnarray*}
          \PP_{\sqrt{a}-\eps} \Bigl( T_{-} < \infty, \mathcal{A}^c \Bigr)
             \le            \PP_{\sqrt{a}-\frac{\eps}{2}} \Bigl( T_{-} < \infty  \Bigr)
           \le
            \PP_{\sqrt{a} - \frac{\eps}{2}} \Bigl(  T_{\sqrt{a}-\eps}  < \infty \Bigl)   \PP_{\sqrt{a}-\eps} \Bigl( T_{-} < \infty  \Bigr).
\end{eqnarray*}
Both inequalities use the fact that the hitting probability
of any level below the starting place is decreasing in the
starting time of the diffusion $X$. Rearranging this formula we get
\begin{align*}
 \PP_{\sqrt{a}-\eps} (T_{-} < \infty) \leq \frac{\PP_{\sqrt{a}-\eps}  ( T_{-} < \infty, \mathcal{A} )}{\PP_{\sqrt{a} - \frac{\eps}{2}}(T_{\sqrt{a}-\eps}  = \infty)}\; .
\end{align*}
\noindent Lemma \ref{lem:pathcontrol} shows that as
$a\to\infty$ the denominator converges to 1.

\paragraph{Application of the Girsanov formula.}

We will now study the right hand side of
(\ref{pathcontrol}) $\PP_{\sqrt{a}-\eps} ( T_{-} < \infty,
\mathcal{A} )$ which is the probability of an event under
which the absolute value of the diffusion is bounded by
$\sqrt{a}$.

In order to find an upper bound for the term
$\eqref{eqstar}$ with $l=\eps$, it would be useful to have
a bound on the time $T_-$. As we do not have any
information about that yet, one idea is to choose the
function $\varphi$ such that the coefficient appearing in
front of the $\sqrt{a}\, T_-$ term becomes negative. The
function $\varphi_1$ of Section \ref{subsec:afirstresult}
works and it gives an upper bound that is sharp up to, but
not including, the exponent of the polynomial factor. This
is the content of Lemma \ref{lessprecise}. As $a\to\infty$
we conclude
\begin{align}
\label{formulelessprecise}
\PP_{\sqrt{a}-\eps}(T_{-} < \infty)  \leq
\exp
\left(-\frac{2}{3}\beta a^{3/2} +
 O(\ln a)\right),
\qquad \eps = \mbox{$\frac{4}{\sqrt\beta}$}\sqrt{\ln
a}/\sqrt[4]{a}.
\end{align}
In addition, Lemma \ref{lessprecise} also shows that with
$\xi=c_3\ln a/\sa$ we have
\begin{align}
\label{e:timecontrol} \PP_{\sqrt{a}-\delta}(\xi\le
T_{-\sa+\delta} < \infty) \leq \exp \left(-\frac{2}{3}\beta
a^{3/2} -\beta\ln a \right), \qquad \delta = \sqrt[4]{\ln
a/a}
\end{align}
i.e.\ long times have polynomially smaller probability than
what we expect for normal times.

\subsection{Final upper bound inside the parabola}

\paragraph{Decomposition according to the time the
process spends near $\sqrt{a}$.}

Let us introduce the last passage time to the level
$\sqrt{a} - \delta$:
\begin{align*}
L := \sup\{t \geq 0 \;:\; X_t = \sqrt{a}-\delta\}, \qquad \delta = \sqrt[4]{\ln
a/a}.
\end{align*}
and use the temporary notation $\tau=c \ln a/\sa$. We can
use the less precise result and, similarly to the part
above the parabola, make a change of variable $\hat{a} := a
+ \tau$. The strong Markov property and the monotonicity
gives
\begin{align}
\PP_{\sqrt{a}- \delta}\left(T_{-\sqrt{a}+\delta} < \infty,
L > \tau\right) \le \PP_{(\tau, \sqrt{a} - \delta)}
\left(T_{-\sqrt{a}+\delta} < \infty\right).
\end{align}
Since $\sqrt{a} - \delta \geq \sqrt{\hat{a}} - \hat{\eps}$
we get the upper bound
\begin{align}
\PP_{(0,\sqrt{\hat{a}} - \hat{\eps})}
\left(\tilde{T}_{-\sqrt{\hat{a}}+\hat{\eps}} <
\infty\right) \leq \exp\left(-2/3\beta a^{3/2} - \beta \ln
a\right) \notag
\end{align}
as long as $c$ (depending on $\beta$) is large enough. The
last inequality for some constant $c
> 0$ follows from the preliminary bound
(\ref{formulelessprecise}). This again is polynomially
smaller than the probability we expect for the main event.

For finer information about the last passage time,
one can divide the paths according to the value of this
last passage time to formalize the idea the process does
not earn a lot when it stays near $\sqrt{a} - \delta$.
\begin{align*}
\PP_{\sqrt{a}- \delta}(T_{-\sqrt{a}+\delta} < \infty,\; L <
\tau) &\le  \sum_{k=0}^{\left\lfloor c\ln a \right\rfloor}
\PP \left(L \in [\frac{k}{\sqrt{a}},
\frac{k+1}{\sqrt{a}}),\; T_{-\sqrt{a}+\delta} <
+\infty\right)
\end{align*}
The event in the sum implies that $X$ visits $\sqrt a- \delta$ in the time interval but not later. By the strong Markov property for the first visit after time $k/\sqrt a$ in that time interval and monotonicity, the sum can be bounded above by
$$
\left\lfloor 1+c\ln a
\right\rfloor\PP_{\sqrt{a}-\delta}\left(L <
\frac{1}{\sqrt{a}},\; T_{-\sqrt{a}+ \delta} <
\infty\right).
$$
To complete the picture, we find an upper bound of the
process between the times $0$ and $1/\sqrt{a}$. This will
be possible thanks to a comparison with  reflected
Brownian motion. Indeed, the drift is non-positive above
the critical parabola, and so up to time $1/\sqrt a$ the process $X_t$ started at $\sqrt{a+1/\sqrt a}$ is stochastically dominated by $\sqrt{a+1/\sqrt a}$ plus reflected Brownian motion.
This leads to the very rough
estimate:
\begin{align*}
\PP_{\sqrt{a + 1/\sqrt{a}}}\left(\sup_{t \in [0,
1/\sqrt{a}]} X_t > c_2 \sqrt{a}\right)
&\leq P\left(\sup_{s \in [0,1/\sqrt{a}]} |B_s| > (c_2-2) \sqrt{a}\right) \\
&\leq \exp\left(-\frac{1}{2}(c_2-2)^2 a^{3/2}\right).
\end{align*}

If $c_2$ is large enough (precisely if $c_2 > 2/\sqrt{3}
\sqrt{\beta} + 2$), this event becomes negligible compared
to the probability to cross the whole parabola. Therefore,
we can examine the studied probability under the event that
$X_t$ is bounded from above by $c_2 \sqrt{a}$ for $t \leq
1/\sqrt{a}$.

\medskip

\noindent We denote by $\mathcal{C}$ the event under which
the above conditions are satisfied
\begin{align*}
 \mathcal{C} := \left\{L < 1/\sqrt{a}, \; \sup_{t \in [0,1/\sqrt{a}]} X_t < c_2 \sqrt{a}\right\}.
\end{align*}
We have just seen that
\begin{align}\label{majoC}
 \PP_{\sqrt{a}-\delta}\left(T_{-\sqrt{a}+\delta} < \infty\right) \leq
(2c \ln a)\PP_{\sqrt{a}-\delta}\left(T_{-\sqrt{a}+\delta} <
\infty, \;\mathcal{C}\right) + O\left(\exp(-2/3 \beta
a^{3/2} - \beta \ln a)\right).
\end{align}

\paragraph{Application of the Girsanov formula.} We will
apply the Girsanov formula with a function $\varphi =
\varphi_2$ such that it compensates exactly the integral
\begin{align*}
 \int_0^{T'_{-\sqrt{a} + \delta}} \left(\frac {8}{\beta} - 2\right) Y_t - 2 \sqrt{a} \, dt.
\end{align*}

The suitable function $\varphi_2$ blows up at $\sqrt{a}$
and $\sqrt{a}$. Therefore we only use it in the interval
$[-\sqrt{a} + \delta, \sqrt{a} - \delta]$ and set
$\varphi_2 := 0$ outside $[-\sqrt{a}, \sqrt{a}]$. Partially
because of those blowups, this function creates error terms
involving the first passage time to the level $-\sqrt{a} +
\delta$, which, if finite, by \eqref{e:timecontrol} can be
assumed to satisfy $T_{-\sa+\delta} < \xi$. Girsanov's
formula applied to the event  $\left\{T_{\sqrt{a} + \delta}
< \xi,\; \mathcal{C}\right\}$ with (\ref{majoC}) leads to
the fundamental upper bound of Proposition
\ref{lem:fondforminside}:
\begin{align}\label{majomilieu}
\PP_{\sqrt{a}- \delta}\left(T_{-\sqrt{a} + \delta} <
\xi\right) \leq \exp\left(-\frac{2}{3} \beta a^{3/2} -
\frac{3}{8} \beta \ln a + O (\sqrt{\ln a})\right).
\end{align}

\paragraph{Conclusion for the upper bound.} Using
(\ref{majomilieu}) with $\tilde{a}$ as in the inequality
(\ref{majodebut}) of the part above the parabola, we deduce
the upper bound part of \eqref{Statement}.

\subsection{Outline of the lower bound}

The lower bound, as often in the literature, is easier. It
suffices to consider the most probable paths. For the part
above the parabola, we can write the inequality
\begin{align*}
\PP_{\infty}(T_{-\infty} < \infty) \geq
\PP_{\infty}\left(T_{\sqrt{a}} \leq \frac{3}{8} \frac{\ln
a}{\sqrt{a}}\right) \, \PP_{\left(\frac{3}{8} \frac{\ln
a}{\sqrt{a}}, \sqrt{a}\right)}\left(T_{-\infty} <
\infty\right)
\end{align*}
and use Proposition \ref{resuabove} to bound the first factor. The second factor can be bounded below by the following:
\begin{align}
\PP_{\left(\frac{3}{8} \frac{\ln a}{\sqrt{a}},
\sqrt{a}\right)}\left(T_{\sqrt{\tilde{a}} - \tilde{\delta}}
< \frac{1}{\sqrt{a}}\right)\times
 \PP_{\sqrt{\tilde{a}} - \tilde{\delta}}\left(T_{-\sqrt{\tilde{a}} + \tilde{\delta}} < \tilde{\xi}\right)  \times \PP_{\left(\tilde{\xi},-\sqrt{\tilde{a}} +
\tilde{\delta}\right)}(T_{-\infty} < \infty).
\label{remainterm}
\end{align}
where $\tilde{a} := 3/8 \ln a/\sqrt{a} + 1/\sqrt{a}$.

A domination by a Brownian motion with drift permits to
deal with the first term of (\ref{remainterm}) which is of the order $\exp(-
O(\sqrt{\ln a}))$. The middle term can be controlled with
the same event $\tilde{\mathcal{C}}$ introduced for the
upper bound with $a$ replaced by $\tilde{a}$. We apply
Girsanov formula directly with the SDE used for the precise
result of the upper bound to obtain Proposition
\ref{lem:fondforminside}:
\begin{align*}
\PP_{\sqrt{\tilde{a}} -
\tilde{\delta}}\left(T_{-\sqrt{\tilde{a}} + \tilde{\delta}}
< \tilde{\xi}, \, \tilde{\mathcal{C}}\right) \geq
\exp\left(-\frac{2}{3} \beta \tilde{a}^{3/2} - \frac{3}{8} \beta
\ln \tilde{a} + o(\ln \tilde{a})\right) \, \PP_{\sqrt{\tilde{a}}-
\tilde{\delta}}\left(T'_{-\sqrt{\tilde{a}} +
\tilde{\delta}} < \tilde{\xi},
\,\tilde{\mathcal{C}'}\right)
\end{align*}
(recall the ``prime'' notation deals with the ``new''
diffusion $Y$).

A comparison with the solution of a simple differential
equation will show that the solution of the new SDE indeed
has a ``large'' probability to go down to
$-\sqrt{\tilde{a}}+\tilde{\delta}$ before the time
$\tilde{\xi}$. This is the content of Lemma
\ref{completelower}.

We conclude the proof of the lower bound by checking that
the last term of (\ref{remainterm}) is also negligible: the
proof is similar to the study above the parabola and can be
found in Proposition \ref{propounderparabola}. \qed

\section{Above the parabola}\label{aboveparabola}
We show at first that we need a certain amount of time to
reach the level $\sqrt{a}$, typically a time $\tau := 3/8
\ln a/\sqrt{a}$. The following Proposition proves indeed
that the probability the process hits $\sqrt{a}$
significantly before $\tau$ is small, but becomes quite
large if this hitting happens around the time $\tau$.
\begin{propo}\label{resuabove}
The following upper bound holds for all sufficiently large  $a$:
\begin{align} \label{aboveresultupper}
\PP_{\infty}\left(T_{\sqrt{a}} \leq \left(\frac{3}{8} -
\frac{1}{\ln a}\right) \frac{\ln a}{\sqrt{a}}\right) \leq
\exp\left(-\frac{4}{3} \beta e^{2 \sqrt{\ln a}}\right).
\end{align}
There exists $c_0 > 0$ depending only on
$\beta$ such that we also have the lower bound:
\begin{align} \label{aboveresultlower}
\PP_{\infty}\left(T_{\sqrt{a}} \leq \frac{3}{8} \frac{\ln
a}{\sqrt{a}}\right) \geq c_0 \frac{1}{\sqrt{\ln a}}.
\end{align}
\end{propo}

\begin{proof}

\noindent {\bf First part.} Fix $\sigma \in (0,1/8)$ and
let
$$\tau' := \left(\frac{3}{8} - \sigma\right) \frac{\ln a}{\sqrt{a}}.$$ It helps to remove the time dependence from the drift coefficient of (\ref{SDE1}). In this direction, consider the SDE:
\begin{align}\label{SDEaboveupper}
\left\{ \begin{array}{ll}
  dY_t &= (a-Y_t^2) dt + \frac{2}{\sqrt{\beta}} dB_t \\
Y_0 &= +\infty.
 \end{array}
\right.
\end{align}

\noindent The process $X$ stochastically dominates
$Y$. Therefore, if $T^Y_{\sqrt{a}}$ is the first passage
time to $\sqrt{a}$ for the diffusion $Y$, we have:
\begin{align*}
\PP_{\infty}\left(T_{\sqrt{a}} \leq \tau'\right) &\leq
\PP\left(T^Y_{\sqrt{a}} \leq \tau'\right).
\end{align*}

Now, we study the difference $Z_t := Y_t -
\frac{2}{\sqrt{\beta}} B_t$ where $B$ is the Brownian
motion driving $Y$ in (\ref{SDEaboveupper}). It satisfies the
(random) ODE:
\begin{align}\label{e:odecomparison}
 dZ_t = \left[a - Z_t^2 \left(1+ \frac{2}{\sqrt{\beta}} \frac{B_t}{Z_t}\right)^2 \right] dt.
\end{align}

\noindent Define $M := \sup\{|B_t|, t \in [0,\tau']\}$ and
notice that thanks to the Brownian tail $(\ref{TailBMup})$,
$$\PP(M \geq 1) \leq 4 \exp\left(-\frac{1}{2 (3/8 - \sigma)} \frac{\sqrt{a}}{\ln a}\right).$$

\medskip

\noindent By definition, for all $t \in [0, \tau' \wedge
T^Y_{\sqrt{a}}]$, the process $Z_t$ is above $\sqrt{a} -
2/\sqrt{\beta} M$ and consequently above the (random)
solution of the differential equation:
\begin{align*}
\left\{ \begin{array}{l}
         F'(t) = a - C f^2(t) \\
F(0) = +\infty
        \end{array} \right.
\end{align*}
where $C$ has the following expression:
$$C := \left(1 + \frac{4}{\sqrt{\beta}} \frac{M}{\sqrt{a} - \frac{2}{\sqrt{\beta}}M}\right)^2.$$

\medskip

\noindent This differential equation admits almost surely
the unique solution
\begin{align*}
\forall t \geq 0, \;F(t) = \sqrt{\frac{a}{C}} \coth(\sqrt{a
C} \, t).
\end{align*}

\noindent Hence,
\begin{align}
\PP\left(T^Y_{\sqrt{a}} \leq \tau'\right) &\leq \PP\left(\inf_{t \in [0,\tau']}\left(F(t) + \frac{2}{\sqrt{\beta}} B_t\right) \leq \sqrt{a}\right) \notag \\
&\leq \PP\left(F(\tau') - \sqrt{a} \leq
-\frac{2}{\sqrt{\beta}} \inf_{t \in [0,\tau']} B_t\right).
\label{compeqdiff}
\end{align}

\noindent Let us compute $F(\tau')$:
\begin{align*}
F(\tau') &= \sqrt{\frac{a}{C}}\; \frac{1+ e^{-2\tau'\sqrt{a
C}}}{1-e^{-2 \tau' \sqrt{a C}}}  = \sqrt{\frac{a}{C}}\;
\left(1 + 2 e^{-2 \tau' \sqrt{a C}} + O\left(e^{-4 \tau'
\sqrt{a C}}\right)\right).
\end{align*}

\noindent Under the event $ \{M \leq 1\}$,
\begin{align*}
\sqrt{C} = 1 + \frac{2}{\sqrt{\beta}} \frac{M}{\sqrt{a}} +
O\left(\frac{1}{a}\right).
\end{align*}

\noindent This implies:
\begin{align}
\sqrt{\frac{a}{C}} = \sqrt{a} - \frac{2}{\sqrt{\beta}} M +
O\left(\frac{1}{\sqrt{a}}\right),
 \label{eq1sqrt}
\end{align}
\noindent and
\begin{align}
\exp\left(- 2 \tau' \sqrt{a C}\right) &= \exp\left(- 2 (3/8 - \sigma) \ln a \left(1 + O\left(\frac{1}{\sqrt{a}}\right)\right)\right) \notag \\
&= \frac{1}{a^{\frac{3}{4} - 2 \sigma}} + O\left(\frac{\ln
a}{a^{\frac{5}{4} - 2 \sigma}}\right). \label{eq2expo}
\end{align}

\noindent Taylor expansions $(\ref{eq1sqrt})$ and
$(\ref{eq2expo})$ give:
\begin{align}
F(\tau') &= \left(\sqrt{a} -\frac{2}{\sqrt{\beta}} M +
O\left(\frac{1}{\sqrt{a}}\right)\right)
\left(1 + \frac{2}{a^{3/4 - 2 \sigma}} + O\left(\frac{\ln a}{a^{5/4 - 2 \sigma}}\right)\right) \notag \\
&= \sqrt{a} - \frac{2}{\sqrt{\beta}} M + \frac{2}{a^{1/4- 2
\sigma}} + O\left(\frac{1}{\sqrt{a}}\right).
\label{devtf(T)}
\end{align}

\noindent Inequality $(\ref{compeqdiff})$ becomes:
\begin{align*}
\PP_{\infty}\left(T_{\sqrt{a}} \leq \tau'\right)
&\leq \PP\left(F(\tau') - \sqrt{a} \leq \frac{2}{\sqrt{\beta}} M, \;M \leq 1\right) + \PP(M > 1) \\
&\leq \PP\left(\sqrt{\beta}\frac{1}{a^{1/4- 2 \sigma}} +
O\left(\frac{1}{\sqrt{a}}\right)\leq M, \;M \leq 1\right)
+ \PP(M > 1) \\
&\leq 4 \exp\left(-\frac{4}{3} \beta\frac{a^{4\sigma}}{\ln
a} + o\left(\frac{a^{4\sigma}}{\ln a}\right)\right).
\end{align*}

\noindent If we take $\sigma = \frac{1}{\sqrt{\ln a}}$, we
have:
\begin{align*}
\PP_{\infty}\left(T_{\sqrt{a}} \leq \left(\frac{3}{8} -
\frac{1}{\sqrt{\ln a}}\right) \frac{\ln a}{\sqrt{a}}
\right) \leq 4 \exp\left(-\frac{4}{3} \beta e^{4 \sqrt{\ln
a} - \ln \ln a}\right),
\end{align*}
and so inequality (\ref{aboveresultupper}) holds.

\paragraph{Second part.} The proof is similar. Let us check the
main lines. At first, we have:
\begin{align}
\PP_{\infty}\left(T_{\sqrt{a}}\leq \tau\right) &\geq
\PP_{\infty}\left(T_{\sqrt{a} + \frac{15}{\sqrt[4]{a}}}
\leq \tau - \frac{1}{\sqrt{a}}\right) \times \PP_{\sqrt{a}
+ \frac{15}{\sqrt[4]{a}}}\left(T_{\sqrt{a}} \leq
\frac{1}{\sqrt{a}}\right). \notag
\end{align}
Now the process $(X_t,\;t\in[0,T_{\sqrt{a}}])$ starting at
a value above $\sqrt{a}$ has a non-positive drift and is
therefore stochastically dominated by Brownian motion.
Thus the second factor can be bounded below by
\begin{align*}
P\left(B_{\frac{1}{\sqrt{a}}} \geq
\frac{15}{\sqrt[4]{a}}\right) =P(B_1 \geq
15). \notag
\end{align*}
For the first factor, instead of the SDE (\ref{SDEaboveupper}) we choose:
\begin{align}\label{SDEabovelower}
\left\{ \begin{array}{ll}
  dY_t &= (a+\tau-Y_t^2) dt + \frac{2}{\sqrt{\beta}} dB_t \\
Y_0 &= +\infty
 \end{array}
\right.
\end{align}
and study the difference $$Z_t := Y_t -
\frac{2}{\sqrt{\beta}} B_t.$$ Set $M' := \sup\{B_t,\; t\in
[0,\tau- 1/\sqrt{a}]\}$. For every $t \in [0, (\tau-
1/\sqrt{a}) \wedge T^Y_{\sqrt{a}}]$, the process $Z_t$ is
below
$$F(t) := \sqrt{\frac{a + \tau}{C}} \coth\left(\sqrt{(a + \tau) C}\, t\right)$$ where
$$C := 1 - \frac{4}{\sqrt{\beta}} \frac{M'}{\sqrt{a} - \frac{2}{\sqrt{\beta}}M'}.$$
The Taylor expansion of $F(\tau-1/\sqrt{a})$ under $\{M'
\leq 1\}$ gives:
\begin{align*}
F\left(\tau-\frac{1}{\sqrt{a}}\right) = \sqrt{a} + \frac{2
e^2}{a^{1/4}} + \frac{2}{\sqrt{\beta}} M' +
O\left(\frac{1}{\sqrt{a}}\right).
\end{align*}
Therefore
\begin{align}
\PP_{\infty}\left(T_{\sqrt{a} + \frac{15}{\sqrt[4]{a}}} \leq
\tau - \frac{1}{\sqrt{a}}\right) &\geq
P\left(F(\tau-\frac{1}{\sqrt{a}}) -
\frac{2}{\sqrt{\beta}} B\left(\tau-
\frac{1}{\sqrt{a}}\right) \leq \sqrt{a} +
\frac{15}{\sqrt[4]{a}}\right). \notag
\end{align}
Since $M'-B(\tau - 1/\sqrt{a})$ has the same law as the
reflected Brownian motion at time $\tau - 1/\sqrt{a}$ and
$15 - 2 e^2 > 0$, we get the  lower bound
\begin{align}
P\left(\frac{2}{\sqrt{\beta}} \left|B\left(\tau -
\frac{1}{\sqrt{a}}\right)\right| \leq \frac{15-2
e^2}{\sqrt[4]{a}}\right) \geq c_0 \frac{1}{\sqrt{\ln a}}.
\label{defc1}
\end{align}
Here $c_0$ in represents an adequate constant depending
only on $\beta$.
\end{proof}

\section{Inside the parabola}
The exponential cost comes from this stretch. This section will
be devoted to the proof of the following Proposition:
\begin{propo}\label{resuinside}
Recall
$\delta := \sqrt[4]{\ln a}/\sqrt[4]{a}$, we have:
\begin{align*}
\PP_{\sqrt{a}-\delta}(T_{-\sqrt{a}+\delta} < \infty) =
\exp\left(-\frac{2}{3} \beta a^{3/2} - \frac{3}{8} \beta
\ln a + O(\sqrt{\ln a})\right).
\end{align*}
\end{propo}

\subsection{Control of the path behavior}

Here we show a lemma about the return to $-\sqrt{a} +
\eps$.
\begin{lemma}\label{lem:pathcontrol}
For $\eps := \frac{4}{\sqrt{\beta}}\sqrt{\ln
a}/\sqrt[4]{a}$ as $a\to\infty$  we have
$\PP_{\sqrt{a}-\eps/2} ( T_{\sqrt{a}-\eps} = \infty) \to
1.$
\end{lemma}
\begin{proof}

Certainly, the probability that $X$ begun at $\sqrt{a} -
\eps/2$ never reaches $\sqrt{a} - \eps$ is bounded below by
the same probability where $X$ is replaced by its reflected
(downward) at $\sqrt{a} - \eps/2$ version. Further, when
restricted to the space interval $[\sqrt{a} - \eps,
\sqrt{a} - \eps/2]$, the $X$-diffusion has drift everywhere
bounded below by $t + 1/2 \sqrt{a} \,\eps$. Thus, we may
consider instead the same probability for the appropriate
reflected Brownian motion with quadratic drift.

\medskip

\noindent To formalize this, it is convenient to shift
orientation. Let now
$$\bar{X} := \frac{2}{\sqrt{\beta}}
B(t) - \frac{1}{2} t^2 - q\, t, \qquad q:= \frac12
\sqrt{a}\,\eps.
$$
Let $X^*$ denote reflected (upward)
at the origin. Namely,
\begin{align}\label{reflectedprocessX*}
X^*(t) = \bar{X}(t) - \inf_{s \leq t} \bar{X}(s).
\end{align}

\noindent If we can show that $P(X^* \mbox{ never reaches }
\eps/2)$ tends to $1$ when $a \to \infty$, then it
will also be the case for the $X$-probability in question.

\medskip

We need to introduce the first hitting time of level $y$
for the new process $\bar{X}$: $\tau_y := \inf\{t \geq 0
\::\: \bar{X}(t) = y\}$. For each $n \in \N$, define the
event:
\begin{align*}
 D_n = \left\{\bar{X}(t) \mbox{ hits } -(n-1) \eps/4 \mbox{ for some } t \mbox{ between } \tau_{-n \eps/4}
\mbox{ and } \tau_{-(n+1) \eps/4}\right\}.
\end{align*}

\noindent From the representation
(\ref{reflectedprocessX*}), one can see that $\{\sup_{t
\geq 0} X^*(t) > \eps/2\}$ implies that some $D_n$ must
occur. Indeed, for this  $\bar X$ must go above its past
minimum by at least $\eps/2$, so in this case it would have
to retreat at least one level before establishing reaching
a new minimum level among multiples of $\eps/4$. Define the
event
\begin{align*}
\mathcal{E} = \left\{\bar{X}(t) \geq -\frac{1}{2}t^2 -2qt
-1 \right\},
\end{align*}
The event  $\mathcal E^c$ is equivalent to the Brownian
motion  $\frac{2}{\sqrt{\beta}}B(t)$ hitting a line of
slope $-q$ starting at $-1$. Since $|B(t)|$ is sublinear,
this will not happen for large enough $q$. So
\begin{align}\label{Ec}
 P(\mathcal{E}^c) \to 0 \quad \mbox{ when }a \to +\infty.
\end{align}
On $\mathcal{E}$, when the following term under the square root is
positive, we have
\begin{align*}
\tau_{-x} \geq  \sqrt{2}\sqrt{2 q^2+x-1}-2 q
\end{align*}
Assume that $q\ge 1$. A calculation shows that if $q\le
\sqrt{x}/2$ then the above inequality implies $\tau_{-x}
\ge \sqrt{x}/2$. So on $\mathcal E$ for all $x\ge 0$ we
have
$$\tau_{-x}+q\ge q\vee \sqrt{x}/2,$$
and so for all $t\ge 0$
\begin{equation}
  \label{process}
 \bar{X}(\tau_{-x} + t)-\bar{X}(\tau_{-x}) \le \frac{2}{\sqrt{\beta}} B(\tau_{-x}+t) - (q\vee \sqrt{x}/2)t.
\end{equation}
Setting $x=\eps n/4$ we see that for all $n \geq 0$
\begin{align*}
P\left(D_n \cap \mathcal{E}\,|\, \mathcal{F}_{\tau_{-n
\eps/4}}\right) &\leq P\left(\mbox{the process on the right
of \eqref{process} hits } \eps/4 \mbox{ before }
-\eps/4\right).
\end{align*}
The distribution of that process is just Brownian motion
with drift. By a stopping time argument for the exponential
martingale $\exp(\gamma \hat B_t - t\gamma^2/2 )$ with
$\gamma = (q\vee \sqrt{n\eps}/4)\sqrt{\beta}$ the above
probability equals $1/(1+e^{\gamma\eps\sqrt{\beta}/8})\le
e^{-\gamma\eps/\sqrt{\beta}/8}$. We
get
\begin{align*}
P(D_n \cap \mathcal{E}) \leq \exp\left(-(q\vee
\sqrt{n\eps}/4){\beta}\eps/8\right).
\end{align*}
Recall that \begin{align*} P\left(\sup_{t \geq 0} X^*(t) >
\eps/2\right) &\leq P(\mathcal{E}^c) + \sum_{n \ge 0}
P(\mathcal{E} \cap D_n).
\end{align*}
The sum of terms where $n\eps \le (4q)^2$ is bounded above
by $ (1+(4q)^2{\eps}^{-1})e^{-\eps q{\beta}/8 }. $ The sum
of the rest is not more than
$$
\sum_{n\ge (4q)^2/\eps} \exp\left( -\eps^{3/2}\beta
\sqrt{n}/32\right)\le c_1 \frac{q}{\eps^{2}}\exp\left(
-\eps q\beta/8\right),
$$
where $c_1$ depends on $\beta$ only. Replacing $\eps$ and
$q$ by their expressions in terms of $a$, and using
\eqref{Ec} we obtain:
\[
P\left(\sup_{t \geq 0} X^*(t) > \eps/2\right) \leq
P(\mathcal{E}^c) +  a^{3/4+o(1)} e^{- ({\beta}/16
)(16/\beta)\ln a} =o(1). \qedhere
\]\end{proof}

\subsection{Application of the Girsanov
formula}\label{subsubsec:Girsanovformula}

\noindent For every $\varphi \in C^2(\R,\R)$ such that
$\sup_{x \in \R} |\varphi(x)| \leq \sqrt{a}$ (the function
$\varphi$ is in fact small compared to the other terms, and
it will be chosen after), we would like to consider the
following SDE (defined on the same probability spaces as
(\ref{SDE1})):
\begin{align}\label{SDE2}
dY_t = \left(-a + Y_t^2 - t + \varphi(Y_t)\right) dt +
\frac{2}{\sqrt{\beta}} dB_t
\end{align}

\begin{rem}
The drift of this SDE is the reversal of the drift in the
initial SDE (\ref{SDE1}). The solution of the new SDE
starting around $\sqrt{a}$ is a good candidate for the
process $X$ conditioned to blow up to $-\infty$ when $a$
goes to $+\infty$. Its expression comes from minimizing
(approximately) the potential:
\begin{align*}
\int_{0}^{s}(g'(u) - (u+a-g^2(u))^2) du.
\end{align*}
over the set of functions $g$ such that $g(0) = \sqrt{a}
\mbox{ and } g(s) = - \sqrt{a}$.
\end{rem}

If we look at events under which the diffusion is bounded,
it is easy to modify the drift of the new diffusion outside
the studied domain and prove that the Novikov condition is
satisfied as long as the examined events are in the space
$\mathcal{F}_{t}$ for a fixed $t > 0$. Let us fix a time $t
> 0$, a level $l \in (0,1)$ and denote by $T_{\pm} :=
T_{\pm (\sqrt{a} - l)}$ the first passage times to $\pm
(\sqrt{a} - l)$ for the diffusion $X$ (respectively
$T'_{\pm} :=  T'_{\pm (\sqrt{a} - \delta)}$ for $Y$). We
take an event $E \in \mathcal{F}_{t} \cap
\mathcal{F}_{T'_-}$ under which the paths of the diffusion
are bounded by a deterministic value, which can depend on
$a$ and $\beta$. Since $E$ is $\mathcal F_t$-measurable,
Girsanov's theorem gives
\begin{align}
\PP_{\sqrt{a} - l}(E) &= \EE_{\sqrt{a} - l}\left(1_{E}
\exp(G_{t}(Y))\right). \notag
\end{align}
The Radon Nikodym density $\exp(G_{\cdot \wedge t\wedge
T'_-}(Y))$ is a bounded martingale, and $E$ is
$\mathcal{F}_{T'_-}$-measurable, so by the optional
stopping theorem the above quantity equals
\begin{align}
\EE_{\sqrt{a} - l}\Big(\EE_{\sqrt{a} - l}\Big(1_E
\exp(G_{t}(Y))\Big{|}
  \mathcal{F}_{T'_-}\Big)\Big)=\EE_{\sqrt{a} - l}\left(1_{E}\;\exp\left(G_{T'_-\wedge
t}(Y)\right)\right). \label{Girsanovwitht}
\end{align}
In the following, we consider $E$ of the form $E = \{T_- <
\infty\}\cap E_1$. The assumption on $E$ requires $E_1$
being an event under which the diffusion is bounded from
above. Taking the limit $t\to +\infty$ leads to the
fundamental formula:
\begin{align}
\PP_{\sqrt{a} - l}(T_{-} < \infty, E_1) = \EE_{\sqrt{a} -
l}\left(1_{T'_-< \infty,\;
E'_1}\;\exp\big(G_{T'_-}(Y)\big)\right).
\end{align}

\noindent Thanks to It\^o's formula we can write  the
exponential martingale $\frac{4}{\beta}G_{T'_-}(Y)$ as
\begin{align}
& 2 \int_0^{T'_-} (t+a-Y_t^2) \,dY_t \; \label{star} \\
& \quad+\phi(Y_0) - \phi(Y_{T'_-}) \notag \\
&\quad + \int_0^{T'_-} \frac{2}{\beta}
\varphi'(Y_t) + \frac{1}{2} \varphi(Y_t)^2 + \varphi(Y_t)
(Y_t^2 - a - t) \,dt,
\label{starstar}
\end{align}
where  $\varphi'$ denotes the derivative of $\varphi$ and
$\phi$ the indefinite integral. Again by It\^o's formula,
we can compute the first term \eqref{star} above.
\begin{align*}
\mbox{\eqref{star}} = 2 a (Y_{T'_-}- Y_0) - \frac{2}{3} (Y_{T'_-}^3 -
Y_0^3) + \left(\frac{8}{\beta} - 2\right) \int_0^{T'_-} Y_t
dt + 2 T'_- Y_{T'_-}.
\end{align*}

\noindent Replacing $Y_{T'_-}$ by its value, we obtain the
expression (valid under $E$):
\begin{align}
\mbox{\eqref{star}} = - \frac{8}{3} a^{3/2} - \frac{4}{3}\, l^3 + 4
\sqrt{a}\, l^2 + 2 \,l\, T'_- - 2 \sqrt{a} T'_- +
\left(\frac{8}{\beta} -2\right) \int_0^{T'_-} Y_t dt.
\label{eqstar2}
\end{align}

\subsection{The preliminary upper bound}\label{subsec:afirstresult}

\noindent Let $c_1$ be a constant such that $c_1 >
(|8/\beta - 2| - 2)\vee 0$. Recall that
$$
\delta := \sqrt[4]{\ln a/a}, \qquad \eps =
\frac4{\sqrt{\beta}}\sqrt{\ln a}/\sqrt[4]{a}
$$
and take the function $\varphi_1$ defined by
\begin{align}\label{def1phi}
\begin{array}{lccc}
\varphi_1 : x \mapsto \left \{ \begin{array}{cl}
 \frac{c_1\,\sqrt{a}}{a - x^2} & \mbox{ if } x \in (-\sqrt{a}+ \delta,\sqrt{a}-\delta) \\
 0 & \mbox{ if } x \not\in (-\sqrt{a},\sqrt{a})
 \end{array} \right.
\end{array}
\end{align}
and extend this function on the entire real line such that
$\varphi_1$ remains a smooth function supported on
$[-\sqrt{a},\sqrt{a}]$ (this is possible for all large
enough ``$a$''). Of course, there are many functions
satisfying the above conditions but we just need to fix
one. Let $\phi_1$ be an antiderivative of $\varphi_1$.

A first step is to prove a less precise upper bound which
does not give us the constant in front of the logarithm
term of (\ref{Statement}). In this subsection, the
notations $Y$, $T'_{-}$, $\mathcal{A'}$ will always refer
to the definitions using this particular $\varphi_1$. We
have the events
$$
\mathcal A = \{T_{\eps-\sa}<T_{\sa-\eps/2}\}, \qquad
\mathcal C= \{T_{\delta-\sa}<T_{c_2\sa},L\le\frac1\sa\}.
$$
\begin{lemma}\label{lessprecise}
(a)
The following inequality holds:
\begin{align}\label{resulessprecise}
\PP_{\sqrt{a}-\eps}(T_{\eps-\sqrt{a}} < \infty,\;
\mathcal{A})
 \leq \exp\left(-\frac{2}{3}\beta a^{3/2} + O(\ln
 a)\right).
\end{align}
(b) for some $c_3\ge 1$ we also have
\begin{align}
\PP_{\sqrt{a}-\delta}\big(c_3 \ln a/\sa \le T_{\delta-\sqrt
a} < \infty,\; \mathcal{C}\big)
 \label{resulesspreciseb}\leq \exp\left(-\frac{2}{3}\beta a^{3/2} -\beta \ln a\right).
\end{align}
\end{lemma}
Part (a) with Lemma \ref{lem:pathcontrol} immediately give the
following.
\begin{corollary}
 The following upper bound holds:
\begin{align*}
\PP_{\sqrt{a}-\eps}(T_{-} < \infty) \leq
\exp\left(-\frac{2}{3}\beta a^{3/2} + O(\ln a)\right).
\end{align*}
\end{corollary}

\begin{proof}
\noindent {\bf Part (a).} First let $T_-=T_{\sa-\eps}$.
Consider the process $Y$ defined with the function
$\varphi_1$. The equality (\ref{Girsanovwitht}) gives:
\begin{align*}
\PP_{\sqrt{a}-\eps}\Big(T_{-} < \infty,\; \mathcal{A}\Big)
= E_{\sqrt{a}-\eps}\Big(1_{\{T'_{-}<
\infty\;\mathcal{A'}\}}\;\exp\big(G_{T'_-}(Y)\big)\Big)
\end{align*}
with $G_{T'_-}(Y)$ given by (\ref{star}-\ref{starstar}).

\medskip
To find an upper bound of the last term \eqref{starstar} in
$G$, we remark that the chosen function $\varphi_1$ on
$[-\sqrt{a}+\eps, \sqrt{a}-\eps/2]$ attains its maximum at
time $\sqrt{a} - \eps/2$. Under the event $\mathcal{A'}$,
it gives:
\begin{align}
\mbox{\eqref{starstar}} \leq \left(\frac{2 c_1^2 + 8/\beta
c_1}{\eps^2} - c_1 \sqrt{a}\right) T'_-.
\label{ineqstarstar}
\end{align}

\noindent Moreover:
\begin{align}
\phi_1(Y_0) - \phi_1(Y_{T'_-}) = c_1\,
\ln\left(2\frac{\sqrt{a}}{\eps} + 1\right).
\label{eqprimitive}
\end{align}

\noindent Thanks to $(\ref{eqstar2})$,
$(\ref{ineqstarstar})$ and $(\ref{eqprimitive})$, we
deduce:
\begin{align*}
\frac{4}{\beta} G_{T'_-}(Y) + \frac{8}{3} a^{3/2} & \leq 4
\sqrt{a}\, \eps^2 + 2 \eps T'_- +
\left(\left|\frac{8}{\beta} - 2\right|-2 - c_1\right)\sqrt{a} T'_-\\
& \quad +  \left(\frac{8}{\beta} c_1 + 2
c_1^2\right)\frac{1}{\eps^2} T'_- +  c_1 \,
\ln\left(2\frac{\sqrt{a}}{\eps}+1\right).
\end{align*}

\noindent We now take $c_1$ such that $|8/\beta - 2| - 2 -
c_1 < 0$ and the coefficient in front of $\sqrt{a} T'_-$
becomes negative and dominates the terms involving $T'_-$.
The last one creates the logarithmic error, and
\eqref{resulessprecise} follows from
\[
 G_{T'_-}(Y) \leq -\frac{2}{3} \beta a^{3/2} + \left(\frac{3}{16}c_1\beta +16\right)\ln a + o(\ln
 a).
\]
\medskip

\noindent {\bf Part (b).} Now let $T_-=T_{\sa-\delta}$. Just
as in part (a), we need to bound the Girsanov terms. To
find an upper bound of the last term \eqref{starstar},
namely
$$
\int_0^{T'_-} \frac{2}{\beta} \varphi'(Y_t) + \frac{1}{2}
\varphi(Y_t)^2 + \varphi(Y_t) (Y_t^2 - a - t) \,dt
$$
note that $\varphi^2$ and $\varphi'$ are both uniformly
$o(\sa)$. On the other hand, we have $\varphi(Y_t)(a-Y^2)$
is nonnegative. Moreover, it is greater than $c_1
\sa+o(\sa)$ as long as $ L'\le t\le T'_{_-}$. So on $\mathcal
C'$ we have the lower bound
\begin{align}
\mbox{\eqref{starstar}} \leq o(\sa)T_-'- c_1 \sqrt{a}(T'_- - L').
\label{ineqstarstarb}
\end{align}
The other inequalities are similar to part $(a)$ with
$\delta$ replaced by $\eps$, except the last term in
\eqref{eqstar2} gives an extra term due to the fact that
$Y$ is only bounded by $c_2\sqrt a$ up to time $L'$.

\noindent Thanks to $(\ref{eqstar2})$,
$(\ref{ineqstarstarb})$ and $(\ref{eqprimitive})$, we
deduce:
\begin{align*}
\frac{4}{\beta} G_{T'_-}(Y) + \frac{8}{3} a^{3/2} & \leq 4
\sqrt{a}\, \delta^2 + 2 \delta T'_- +
\left(\left|\frac{8}{\beta} - 2\right|-2 - c_1\right)\sqrt{a} T'_-\\
& \quad + (c_1 + c_2)\sa L'+o(\sa) T'_- +  c_1 \,
\ln\left(2\frac{\sqrt{a}}{\delta}+1\right).
\end{align*}

\noindent From part (a), we have $c:=|8/\beta - 2| - 2 -
c_1<0$. We chose $c_3\ge 1$ large enough so that the terms
involving $T'_-\ge c_3 \ln a/a$ together with the $\ln a$
term coming from the antiderivative are less than $-\beta \ln a$, i.e.\
$\frac{\beta}{4} c_3 c+\frac{3}{16}c_1\beta<-\beta$. This completes the
proof of \eqref{resulesspreciseb} since
\[
 G_{T'_-}(Y) \leq -\frac{2}{3} \beta a^{3/2} -\beta \ln a + o(\ln
 a). \qedhere
\]\end{proof}

\subsection{Precise asymptotics for the exponent}\label{logprecision}

Recall $\delta := \sqrt[4]{\ln a/a}$, $L$ the
last passage time to $\sqrt{a} - \delta$, and the event
introduced for technical reasons:$$\mathcal{C} := \left\{L
< 1/\sqrt{a},\;\; \sup_{t \in [0,1/\sqrt{a}]} X_t < c_2
\sqrt{a}\right\}$$ defined in the outline of the proof. We
will study in this section $\PP_{\sqrt{a}-
\delta}\left(T_{-\sqrt{a}+\delta} < \infty,
\;\mathcal{C}\right)$.

In order to obtain the coefficient in front of the
logarithm term, we need to be more precise in our analysis
and we will look more carefully at $T_-$, the first passage
time to $-\sqrt{a} + \delta$ of $X$.

Our tool is again the Cameron-Martin-Girsanov formula with
a drift containing a different function $\varphi$. Let us
define $\varphi_2$ in the following way:
\begin{align}\label{def2phi}
\begin{array}{lccc}
\varphi_2: x \mapsto \left \{ \begin{array}{cl}
 \frac{(8/\beta-2) x -2 \sqrt{a}}{a - x^2} & \mbox{ if } x \in (-\sqrt{a}+ \delta/2,\sqrt{a}- \delta) \\
 0 & \mbox{ if } x \not\in (-\sqrt{a},\sqrt{a})
 \end{array} \right.
\end{array}
\end{align}
and extend it such that it remains a smooth function on
$\R$ satisfying: $\sup |\varphi| \leq \sqrt[4]{a}$ and
$\sup |\varphi'| \leq \sqrt{a}$ (this is possible for a
large enough ``$a$''). Similarly to the previous
subsection, the notations $Y$, $T'_-$, $\mathcal{C}'$ etc.
refers to definitions with this chosen function.

\begin{propo}\label{lem:fondforminside} We have
\begin{align*}
&\PP_{\sqrt{a}- \delta}\left(T_- < c_3\ln a/\sa,
\;\mathcal{C}\right) \\& \qquad = \exp\Big(-\frac{2}{3}
\beta a^{3/2} - \frac{3}{8} \beta \ln a + O\big(\sqrt{\ln
a}\big)\Big) \PP_{\sqrt{a}- \delta}\left(T'_- < c_3\ln
a/\sa, \;\mathcal{C}'\right)
\end{align*}
\end{propo}
\begin{proof}
\noindent Let us compute the new Radon-Nikodym derivative
according to the position of $Y$ using the relations
(\ref{star}-\ref{starstar}) and (\ref{eqstar}).
\medskip

\noindent At first, the term $\phi(Y_0) - \phi(Y_{T'_-})$
is equal to $-3/2 \ln a$. Moreover,
\begin{align*}
\forall y \in \left[-\sqrt{a} + \delta,\sqrt{a}-\delta
\right],\:\: -2 \sqrt{a} + \left(\frac{8}{\beta} - 2\right)
y + (y^2 - a) \varphi(y) =0.
\end{align*}

\noindent Consequently, there exists a constant $c' > 0$,
depending only on $\beta$, such that for every $y \in
[-\sqrt{a} + \delta,\sqrt{a}-\delta]$,
\begin{align*}
\left|-2 \sqrt{a} + \left(\frac{8}{\beta} - 2\right) y +
(y^2 - a) \varphi(y) + \frac{2}{\beta} \varphi'(y) +
\frac{1}{2}\varphi(y)^2\right| &\leq \frac{c' a}{(a-y^2)^2}
\leq \frac{2 c'}{\delta^2}.
\end{align*}
\noindent There is another constant $c^{\prime\prime} > 0$
such that
\begin{align*}
\left|\int_0^{T'_-} u \,\varphi(Y_u) du\right| \leq
\frac{c^{\prime\prime}}{\delta} T'^2_-.
\end{align*}

\noindent For every $y \geq \sqrt{a}-\delta$,
\begin{align*}
\left|\frac{2}{\beta} \varphi'(y) + \frac{1}{2}
\varphi(y)^2 + \varphi(y) (y^2 - a - t)\right| \leq
\left(\frac{2}{\beta} + 1\right) \sqrt{a}
\end{align*}

\noindent Putting all together and using the upper bound on
the last passage time to $\sqrt{a} - \delta$ contained in
$\mathcal{C}$, we obtain:
\begin{align*}
\left|\frac{4}{\beta} G_{T'_-}(Y) + \frac{8}{3} a^{3/2} +
\frac{3}{2} \ln a\right| \leq 4 \sqrt{a} \delta^2 + \frac{2
c'}{\delta^2} T'_- + \frac{c^{\prime\prime}}{\delta} T'^2_-
+ 2 \delta T'_- + \frac{4}{3} \delta^3 +
\left(\frac{2}{\beta} + 1\right).
\end{align*}
\noindent If $\{T'_- \leq c_3\ln a/\sa\}$ holds,
\begin{align*}
\frac{2 c'}{\delta^2} T'_- +
\frac{c^{\prime\prime}}{\delta} T'^2_- + 2 \delta T'_- +
\frac{4}{3} \delta^3 \leq 2 c' \sqrt{\ln a} + O(1).
\end{align*}

\medskip

\noindent Under $\{T'_- < c_3\ln a/\sa\}\cap \mathcal{C}'$,
we conclude
\[
G_{T'_-}(Y) = -\frac{2}{3} \beta a^{3/2} - \frac{3}{8}
\beta \ln a + O\left(\sqrt{\ln a}\right) \qedhere\]
\end{proof}

\noindent To complete the study inside the parabola for the
lower bound, we prove:
\begin{lemma}\label{completelower}
There exists $c_4 >0$ depending only on $\beta$ such that
with $\xi=c_3\ln a/\sa$
\begin{align*}
\PP_{\sqrt{a}-\delta}(T'_- <\xi,\;\mathcal{C}') \geq
\exp\left(-c_4 \sqrt{\ln a}\right).
\end{align*}
\end{lemma}
\begin{proof}For the lower bound we can replace the event
$C'$ by the event that $T'_+=T'_{\sqrt a-\delta/2}$ is
infinite, i.e.\ the corresponding level is never hit. We
will show that this events happens as long as
$$M := \sup\left\{|B_t|,\; t \in \left[0,
\xi\right]\right\} \le \delta\sqrt\beta/5,$$ which by the
Brownian motion estimate (\ref{TailBMlower}) has the right
probability.

Let $\xi= c_3\ln a/\sa$. Again, we compare our equation to
an ODE. The quantity $Z=X-B$ on $[0,\xi]$ satisfies
$$
Z'=-(Z-B)^2+t-a\le Z^2+a+ \mbox{$\frac{4}{\sqrt{\beta}}$}MZ
+\xi,
$$
so let $H$ be the solution of the (random) ODE:
\begin{align}
\left\{ \begin{array}{l}
         H'(t) = H^2(t) - C, \label{ODEH} \\
H(0) = \sqrt{a}-\delta,
        \end{array} \right.
\end{align}
where the random constant satisfies
$$C= a - \frac{4}{\sqrt{\beta}} \sqrt{a} M +O(1).$$

\noindent By the same argument of comparison as in the
Section $\ref{aboveparabola}$, when $M\le c\delta$ the
diffusion $Y$ is under $t \mapsto H(t) + 2/\sqrt{\beta}
B_t$ up to the minimum of $\xi$ and the exit time from
$[-\sa-1,\sa]$. Therefore we will have $T'_- < \xi, T'_- <
T'_+$ as long as
\begin{align*}
H(\xi) + \frac{2}{\sqrt{\beta}}B_\xi\leq -\sqrt{a} +
\delta,\quad \mbox{and} \qquad\sup_{s\in [0,\xi]}
\left(H(s) + \frac{2}{\sqrt{\beta}} B_s\right) < \sqrt{a}
-\delta/2.
\end{align*}
Since $H(s)$ is decreasing in $s$, the second event is
implied by our assumption on $M$.

The solution $H$ takes the form:
\begin{align*}
H(t) = - \sqrt{C} \tanh\left(\sqrt{C} t -
\operatorname{arctanh}\left(b\right)\right) =
\frac{\sqrt{C} \left(\tanh \left(\sqrt{C}
t\right)-b\right)}{b \tanh \left(\sqrt{C} t\right)-1},
\qquad b=\frac{\sqrt{a}-\delta}{\sqrt{C}}.
\end{align*}
When $c_3\ge 1$ we have $\tanh (\sqrt{C} t)=1+O(a^{-2})$. So
we get the asymptotics
$$
H(\xi)=-\sqrt{a}+ 2M/\sqrt{\beta}+o(\delta),
$$
and we indeed have
\[
H(\xi) + \frac{2}{\sqrt{\beta}} B_\xi \le
-\sqrt{a}+\frac{4M}{\sqrt{\beta}} \le \frac45\delta
+o(\delta). \qedhere
\]
\end{proof}

\section{Under the parabola, lower bound}\label{underparabola}

We will prove:
\begin{propo}\label{propounderparabola}
There exists $c_5 > 0$ depending only on $\beta$ such that,
\begin{align*}
\PP_{-\sqrt{a} + \delta} (T_{-\infty}< \infty) \geq
\exp\left(-c_5 \sqrt{\ln a}\right).
\end{align*}
\end{propo}
\begin{proof}
Using the strong Markov property and the increasing
property, we can lower bound the left hand side by
\begin{align*}
\PP_{-\sqrt{a} + \delta}\left(T_{-\sqrt{a}-\delta} <
\frac{1}{\sqrt{a}}\wedge T_{-\sqrt{a}+2\delta}\right)
\times \PP_{(\frac{1}{\sqrt{a}}, -\sqrt{a} -
\eps)}\left(T_{-\sqrt{a}-\frac{\sqrt{\ln a}}{\sqrt[4]{a}}}
< \frac{\ln a}{2 \sqrt{a}} \wedge T_{-\sqrt{a}}\right)
\\ \times \PP_{\left(\frac{\ln a}{\sqrt{a}}, -\sqrt{a} -
\frac{\sqrt{\ln a}}{\sqrt[4]{a}}\right)}\left(T_{-\infty} <
\infty \right).
\end{align*}

\noindent $\bullet$ The first probability gives the main
cost. Under this event, the process $X$ is stochastically
dominated by the drifted Brownian motion: $$t \mapsto -
\sqrt{a} + \delta + 2 \sqrt{a} \,\delta t +
\frac{2}{\sqrt{\beta}} B_t.$$ Thus,
\[
\PP_{-\sqrt{a} + \eps}\bigg(T_{-\sqrt{a}-\eps} <
\frac{1}{\sqrt{a}}\wedge T_{-\sqrt{a}+2\eps}\bigg) \notag
 \geq \PP\left(B_1 < -\frac{3}{2}\sqrt{\beta} \sqrt[4]{a}
\,\eps, \; \sup_{s\in [0,1]} B_s \leq \frac{1}{2}
\sqrt{\beta} \sqrt[4]{a} \eps\right). \notag
\]
By the reflection principle, this equals
\[ \PP\left(\frac{5}{2} \sqrt{\beta}  \sqrt[4]{a} \eps \geq B_1 \geq \frac{3 }{2} \sqrt{\beta} \sqrt[4]{a} \, \eps\right)
\geq \sqrt{\beta} \sqrt[4]{\ln a} \,
\exp\left(-\frac{25}{4}\beta \sqrt{\ln a}\right). \notag
\]

\medskip

\noindent $\bullet$ For the second part, under the studied
event, the diffusion $(X_t,\;t \geq 1/\sqrt{a})$ is
stochastically dominated by
$$\frac{1}{4 \sqrt{a}} + t + \frac{2}{\sqrt{\beta}}B_t -\sqrt{a} - \eps.$$
Thus the studied probability is bounded from below by a
constant depending only on $\beta$.

\medskip

\noindent $\bullet$ For the last part, we need to compare
the diffusion with the solution of a simple differential
equation. Similarly to the previous comparisons, under the
event $\{X(\ln a/\sqrt{a} = - \sqrt{a}-\sqrt{\ln
a}/\sqrt[4]{a},\; T_{-\infty} < T_{-\sqrt{a}}\wedge (3/8
\ln a/\sqrt{a}\}$, the diffusion $X$ is stochastically
dominated by $G(t)+2/\sqrt{\beta} B_t$ where $G$ is the
solution of the differential equation:
\begin{align*}
\left \{
\begin{array}{l}
G'(t) = a + \frac{11}{8} \frac{\ln a}{\sqrt{a}} - \left(1- \frac{4}{\sqrt{\beta}} \frac{M}{\sqrt{a}}\right)G^2(t) \\
G(0) = -\sqrt{a}-\frac{\sqrt{\ln a}}{\sqrt[4]{a}}.
\end{array} \right.
\end{align*}
and $$M = \sup_{s \in [0,\frac{3}{8}\frac{\ln
a}{\sqrt{a}}]} |B_s|.$$ Whenever we have $$\left\{M \leq
\frac{2}{2\,\sqrt{\beta}}\frac{\sqrt{\ln
a}}{\sqrt[4]{a}}\right\},$$ the function $G$ blows up to
$-\infty$ at a time smaller than $3/8 \ln a/\sqrt{a}$ and
the diffusion $(X_t,\; t\in [0, 3/8 \ln a/\sqrt{a}])$ stays
under $-\sqrt{a}$. Therefore:
\begin{align*}
\PP_{\left(\frac{\ln a}{\sqrt{a}}, -\sqrt{a} -
\frac{\sqrt{\ln a}}{\sqrt[4]{a}}\right)}\left(T_{-\infty} <
\infty \right) \geq P\left(M \leq
\frac{2}{2\,\sqrt{\beta}}\frac{\sqrt{\ln
a}}{\sqrt[4]{a}}\right)
\end{align*}
which is greater than a constant depending only on $\beta$.
It leads to the result.
\end{proof}

\noindent \textbf{Acknowledgement:} The authors would like to thank the BME of Budapest for its kind hospitality 
while we finished this article. L.D. is also grateful to the 
mathematics department of the University of Toronto for its welcome during her visits. 
L.D. acknowledges support from the Balaton/PHC grant \#19482NA and 
B. V. Canada Research Chair program and the NSERC DAS program.

\bibliographystyle{dcu}
\bibliography{biblioTW}
\end{document}